\documentclass[11pt,a4paper]{amsart}

\usepackage{amsrefs}

\usepackage[colorlinks=false]{hyperref}

\newtheorem{theorem}{Theorem}[section]
\newtheorem{definition}[theorem]{Definition}

\newtheorem{corollary}[theorem]{Corollary}
\newtheorem{example}[theorem]{Example}

\newcommand{\Dd}{\widehat{\mathbb{D}}}
\newcommand{\codim}{\operatorname{codim}}

\begin{document}

\title{Variations of Weyl's tube formula}

\author[A.~Burtscher]{Annegret Burtscher$^\dagger$}
\author[G.~Heckman]{Gert Heckman$^\dagger$}

\makeatletter
\@namedef{subjclassname@2020}{
  \textup{2020} Mathematics Subject Classification}
\makeatother
 
\thanks{$^\dagger$Department of Mathematics, IMAPP, Radboud University Nijmegen, 
The Netherlands, \textit{Emails}: \texttt{burtscher@math.ru.nl}, \texttt{g.heckman@math.ru.nl}.\\
\indent \emph{Acknowledgements:} Research of the first author supported by the Dutch Research Council
(NWO), Project number VI.Veni.192.208. Part of this material is based upon
work supported by the Swedish Research Council under grant no.\ 2016-06596
while the first author was in residence at Institut Mittag-Leffler in Djursholm,
Sweden in the Fall of 2019.}

\maketitle

\begin{abstract}
In 1939 Weyl showed that the volume of spherical tubes around compact 
submanifolds $M$ of Euclidean space depends solely on the induced Riemannian 
metric on $M$. Can this intrinsic nature of the tube volume be preserved for 
tubes with more general cross sections $\mathbb{D}$ than the round ball? 
Under sufficiently strong symmetry conditions on $\mathbb{D}$ the answer turns 
out to be yes.
\end{abstract}

\section{Introduction}\label{Introduction}

Let us be given a compact connected manifold $M$ (possibly with a boundary) 
of dimension $n$ embedded in $\mathbb{R}^{n+m}$ as submanifold of codimension 
$m$. For each $r\in M$ we have an orthogonal decomposition $T_rM\oplus N_rM$ 
of $\mathbb{R}^{m+n}$ into tangent space and normal space at $r$ of $M$. 
It was shown by Weyl~\cite{Weyl 1939a} that the Euclidean volume 
of the spherical tube
\[ \{r+n \, ; r\in M, \, n\in N_rM, \, |n|\leq a\} \]
around $M$ with radius $a>0$ sufficiently small is equal to
\[ V_M(a)=\Omega_m\;\sum_{d=0}^n\frac{k_d(M)\,a^{m+d}}{(m+2)\cdots(m+d)}
   \qquad (d\;\mathrm{even}) \] 
with $\Omega_m$ the volume of the unit ball 
$\mathbb{B}^m=\{t \, ; |t|\leq1\}$ in $\mathbb{R}^m$. 

The remarkable insight of Weyl is that the coefficients $k_d(M)$ are 
integral invariants of $M$ only determined by the \emph{intrinsic} 
metric nature of $M$. For example, the initial coefficient 
$k_0(M)=\int_M ds$ is the Riemannian volume of $M$ and the next 
coefficient is $k_2(M)=\tfrac12\int_MS\,ds$ with $S$ the scalar 
curvature of $M$. If $M$ has empty boundary and is of even dimension it 
was proved by Allendoerfer and Weil~\cite{Allendoerfer--Weil 1943} in 
their approach towards the Gauss--Bonnet theorem that the top 
coefficient $k_n(M)=(2\pi)^{n/2}\chi(M)$ with $\chi(M)$ the Euler 
characteristic of $M$ is even of topological nature. See also the text 
books of Gray on tubes~\cite{Gray 2004} and of Morvan on
generalized curvatures~\cite{Morvan 2008} for further details.

Due to the \emph{local} nature of the tube formula we can assume that 
the submanifold $M$ of $\mathbb{R}^{n+m}$ comes with a chosen orthonormal 
frame in the normal bundle $NM$ of $M$ in $\mathbb{R}^{n+m}$. In turn this 
gives for all $r\in M$ an identification of the normal space $N_rM$ with 
$\mathbb{R}^m$, and so for $\mathbb{D}^m$ a compact domain 
around $0$ in $\mathbb{R}^m$ we can consider the \emph{generalized tube}
\[ \{r+n \, ; r\in M, \, n\in a\mathbb{D}^m\} \] 
around $M$ of type $a\mathbb{D}^m$ for $a>0$ sufficiently small. The main result of 
this paper is that under sufficiently strong (relative to the dimension 
$n$ of $M$) symmetry requirements on the domain $\mathbb{D}^m$ a similar 
intrinsic formula for the volume $V(a)$ of the above generalized tube 
remains valid as in Weyl's case where $\mathbb{D}^m$ equals the unit ball 
$\mathbb{B}^m$.

Our generalized tubes share the feature that 
the domains $\mathbb{D}^m$ are invariant under the following subgroups of 
the orthogonal group.

\begin{definition}\label{orthogonal of degree n definition}
A subgroup $G_m$ of the orthogonal group $\mathrm{O}_m(\mathbb{R})$ on 
$\mathbb{R}^m$ is called \emph{orthogonal of degree $n$} if any polynomial 
$p(t)\in\mathbb{R}[t]$ on $\mathbb{R}^m$ of degree $\leq n$ that is invariant 
under $G_m$ is in fact invariant under the full orthogonal group 
$\mathrm{O}_m(\mathbb{R})$.
\end{definition}
 
Our principal result is the following generalized tube formula.

\begin{theorem}\label{abstract tube formula}
Let $M$ be a compact connected manifold of dimension $n$ embedded in Euclidean
space $\mathbb{R}^{n+m}$. 
If the compact domain $\mathbb{D}^m$ around $0$ in $\mathbb{R}^m$ has a 
symmetry group $G_m$ inside $\mathrm{O}_m(\mathbb{R})$ that is orthogonal 
of degree $n$ then the volume of the generalized tube of type 
$a\mathbb{D}^m$ for $a>0$ sufficiently small is given by
\[ V_M(a)= \sum_{d=0}^n 
   \frac{\{\int_{\mathbb{D}^m}|t|^d\,dt\}\,k_d(M)\,a^{m+d}}{m(m+2)\cdots(m+d-2)} 
   \qquad (d\;\mathrm{even}) \]
with intrinsic coefficients $k_d(M)=\int_M H_d\,ds$ as specified in 
Theorem~\ref{integrand average theorem}.  
\end{theorem}

In Sections~\ref{The volume of tubes}--\ref{Averaging the integrand} we shall
review the proof of the tube formula following 
Weyl's original approach and along the way obtain variations of the 
tube volume formula for polyhedral (and related) tubes rather 
than spherical tubes. 
We discuss several examples in Section~\ref{Polyhedral examples} and 
counterexamples in Section~\ref{No go results for diamant domains}, 
as well as causal tubes in Minkowski space in 
Section~\ref{Riemannian submanifolds of a Lorentzian vector space}.

After this paper was finished we learned that tube volume formulas with
more general cross sections $\mathbb{D}^m$ had already been studied before
by Domingo-Juan and Miquel 
~\cite{Domingo-Juan--Miquel 2004}. We decided to leave our paper as it was,
but add Section~\ref{Pappus type theorems} in order to
briefly survey their approach and compare their results with ours.
 
\section{The volume of tubes}\label{The volume of tubes}

Locally a submanifold $M$ of dimension $n$ in $\mathbb{R}^{n+m}$ is 
given in the Gaussian approach by a parametrization
\[ r\colon U^n\rightarrow M \subset \mathbb{R}^{n+m},\quad u\mapsto r(u) \]
with $u=(u^1,\ldots,u^n)\in U^n\subset\mathbb{R}^n$ and $\partial_i=\partial/\partial u^i$ 
in the usual notation of differential geometry. A summation sign $\sum$ without explicit
mention of indices always means a summation over all indices, which occur both as 
upper and as lower index. The first fundamental form (or Riemannian metric) is given by
\[ ds^2=\sum g_{ij}du^idu^j,\quad g_{ij}=\partial_ir\cdot\partial_jr, \]
with $\cdot$ the scalar product on the ambient Euclidean space $\mathbb{R}^{n+m}$
and $g_{ij}=g_{ij}(u)$ a positive definite symmetric matrix for all $u\in U^n$. 

Let us choose an orthonormal frame field $u\mapsto n_1(u),\ldots,n_m(u)$ in 
the normal bundle of $M$, and so $\partial_ir\cdot n_p=0$ and 
$n_p\cdot n_q=\delta_{pq}$ along $M$ for all $i=1,\ldots,n$ and 
$p,q=1,\ldots,m$. Let $t=(t^1,\ldots,t^m)$ be Cartesian coordinates on
$\mathbb{R}^m$. Let us be given a compact domain $\mathbb{D}^m$
around $0$ in $\mathbb{R}^m$ such that the map 
\[
 x \colon U^n\times\mathbb{D}^m \to\mathbb{R}^{n+m}, \quad 
 (u,t) \mapsto x(u,t)=r(u)+\sum t^pn_p(u)
\] 
is a diffeomorphism of $U^n\times\mathbb{D}^m$ onto its image in
$\mathbb{R}^{n+m}$. This image is called a tube of type $\mathbb{D}^m$ around $r(U^n)$. 

We are interested in the Euclidean volume $V_{U^n}(a)$ of the local tube 
\[ \{x(u,t)=r(u)+\sum t^pn_p(u) \, ; u\in U^n,\;t\in a\mathbb{D}^m\}
   \subset\mathbb{R}^{n+m} \]
of type $a\mathbb{D}^m$ as a function of a small positive parameter $a>0$. 
By the Jacobi substitution theorem we have
\[ V_{U^n}(a)=\int_{U^n}\{\int_{a\mathbb{D}^m}J(u,t)\,dt\}\,du \]
with $J$ the absolute value of the determinant
\[ \det(\partial_1x\;\cdots\;\partial_nx\;n_1\;\cdots\;n_m ) \]
and $\partial_ix=\partial_ir+\sum t^p\partial_in_p$ for $i=1,\ldots,n$.

Recall that
\[ \partial_i\partial_jr=\sum\Gamma_{ij}^k\partial_kr+
   \sum h_{ij}^pn_p \]
with $\Gamma_{ij}^k=\Gamma_{ij}^k(u)$ the Christoffel symbols and
$h_{ij}^p=h_{ij}^p(u)$ the coefficients of the second fundamental form $h_{ij}$
relative to the orthonormal normal frame $n_p$. Here indices $p,q=1,\ldots,m$
are coordinate indices in the normal direction, while the other indices 
$i,j,k=1,\ldots,n$ are coordinate indices on the submanifold $M$.
Since $\partial_jr\cdot n_p=0$ we get
\[ \partial_in_p\cdot\partial_jr=-n_p\cdot\partial_i\partial_jr=-\sum\delta_{pq}h_{ij}^q \]
with $\delta_{pq}$ the Kronecker symbol. Writing $t_p=\sum\delta_{pq}t^q$ we find
\[ \partial_ix=\partial_ir-\sum \delta_{pq}t^ph_{ik}^qg^{kj}\partial_jr+\ldots=
   \sum(\delta_i^j-\sum t_ph_{ik}^{p}g^{kj})\partial_jr+\ldots \]
with $g_{ij}=\partial_ir\cdot\partial_jr$, $g^{ij}$ its inverse matrix, 
$\det g = \det (g_{ij})$ its determinant and $\ldots$ stands for a linear 
combination of the normal fields $n_p$. If in the usual notation we write 
$h_{i}^{jp}=\sum g^{jk}h_{ik}^p$ for $i,j=1,\ldots,n$ and $p=1,\ldots,m$ then we get
\begin{gather*}
   \det(\partial_1x\;\cdots\;\partial_nx\;n_1\;\cdots\;n_m )= \\
   \det(\delta_i^j-\sum t_ph_{i}^{jp})
   \det(\partial_1r\;\cdots\;\partial_nr\;n_1\;\cdots\;n_m)
\end{gather*}   
which in turn implies
\[ V_{U^n}(a)=\int_{U^n}\{\int_{a\mathbb{D}^m}
   \det(\delta_i^j-\sum t_ph_{i}^{jp})\,dt\}\sqrt{\det g_{ij}} \, du \]
for all $a>0$ sufficiently small. 

For fixed $u\in U^n$ the integrand
$\det(\delta_i^j-\sum t_ph_{i}^{jp})$ is a polynomial in $t$ 
of degree $n$ and so, after integration and patching together the locally 
defined tubes, we conclude that
\[ V_M(a)=\sum_{d=0}^n v_d\,a^{m+d} \]
is a polynomial in $a$ of degree $m+n$ and with coefficient $v_0=
\mathrm{vol}(\mathbb{D}^m)\mathrm{vol}(M)$. If the domain
$\mathbb{D}^m$ is centrally symmetric with respect the origin, that is if
$-\mathbb{D}^m=\mathbb{D}^m$, then the integrals of odd degree
monomials in $t$ over $\mathbb{D}^m$ vanish and so $v_d=0$ for $d$ odd. 

In order to show that the volume $V_M(a)$ of a generalized tube of type
$\mathbb{D}^m$ depends only on intrinsic quantities of $M$, two steps
are necessary. Firstly, by assuming that $M$ is embedded in flat
$\mathbb{R}^{n+m}$, one observes that certain combinations of the 
second fundamental forms are intrinsic curvature quantities (this 
was already done by Weyl). Secondly, by imposing certain symmetry 
conditions on $\mathbb{D}^m$, we show that only those intrinsic combinations 
remain in the volume formula $V_M(a)$ for the generalized tube (done by 
Weyl for the ball $\mathbb{B}^m$). These steps are carried out in 
Sections~\ref{The Gauss equations} and \ref{Averaging the integrand}, 
respectively.

\section{The Gauss equations}\label{The Gauss equations}

As before, we write
\[ \partial_i\partial_jr=\sum \Gamma_{ij}^k\partial_kr+h_{ij} \]
with $\Gamma_{ij}^k$ the Christoffel symbols given by 
\[ \tfrac{1}{2} \sum g^{kl}(\partial_ig_{jl}+\partial_jg_{il}-\partial_lg_{ij}) \] 
and $h_{ij}=\sum h_{ij}^pn_p$ the second fundamental form relative 
to the orthonormal frame $n_p$ in the normal bundle along $M$. 
Given scalar functions $g_{ij}$ and $h_{ij}^p$, the integrability 
conditions for the existence of an embedding of $M$ into flat Euclidean space with these functions 
as coefficients of the first and second fundamental forms are given by
\[ \partial_i(\sum \Gamma_{jk}^l\partial_lr+\sum h_{jk}^pn_p)-
   \partial_j(\sum \Gamma_{ik}^l\partial_lr+\sum h_{ik}^pn_p)=0 \]
for all $i,j,k$ (by working out $\partial_i(\partial_j\partial_kr)-
\partial_j(\partial_i\partial_kr)=0$).

In the normal directions this leads to the Codazzi--Mainardi equations
\[ \partial_ih_{jk}^p-\partial_jh_{ik}^p+
   \sum(\Gamma_{jk}^lh_{il}^p-\Gamma_{ik}^lh_{jl}^p)=0 \]
for all $i,j,k$ and all $p$. In the tangential directions this amounts
to the Gauss equations
\[ R_{kij}^l=\sum \delta_{pq}g^{ln}(h_{in}^ph_{jk}^q-h_{jn}^ph_{ik}^q) \]
for all $i,j,k,l$ with
\[ R_{kij}^l=\partial_i\Gamma_{kj}^l-\partial_j\Gamma_{ki}^l+
   \sum (\Gamma_{kj}^m\Gamma_{mi}^l-\Gamma_{ki}^m\Gamma_{mj}^l) \]
the coefficients of the Riemann curvature tensor. As mentioned earlier, by 
raising indices $h_{i}^{jp}=\sum g^{jk}h_{ki}^p$ and 
$R_{ij}^{kl}=\sum g^{ln}R_{nij}^k$ the Gauss equations take the form
\[ R_{ij}^{kl}=\sum \delta_{pq}(h_{i}^{kp}h_{j}^{lq}-h_{j}^{kp}h_{i}^{lq})=
   h_i^{k}\cdot h_j^l-h_j^k\cdot h_i^l \]
for all $i,j,k,l$ and $h_i^j=\sum h_i^{jp}n_p=\sum g^{jk}h_{ki}$ normal vectors along $M$.

\section{Averaging the integrand}
\label{Averaging the integrand}

For $p(t)\in\mathbb{R}[t_1,\cdots,t_m]$ a polynomial on Euclidean space 
$\mathbb{R}^m$ and $G_m$ a closed subgroup of the orthogonal group 
$\mathrm{O}_m(\mathbb{R})$ let us write
\[ \langle p(t)\rangle_{G_m}=\int_{G_m}\;p(gt)\,d\mu(g)\]
for the average of $p$ over $G_m$, with $\mu$ the normalized Haar measure on $G_m$. 
Clearly
\[ \langle p(t)\rangle_{\mathrm{O}_m(\mathbb{R})}\in\mathbb{R}[t\cdot t] \]
with $t\cdot t=|t|^2$ the norm squared of $t\in\mathbb{R}^m$.
The crucial step for the intrinsic nature of the coefficients of
the tube volume formula is the following result
(see the Lemma on page 470 of Weyl's paper \cite{Weyl 1939a}).

\begin{theorem}\label{integrand average theorem}
We have (with $1\leq i,j\leq n$ and $1\leq p\leq m$)
\[ \langle\det(\delta_i^j-\sum t_ph_{i}^{jp})
   \rangle_{\mathrm{O}_m(\mathbb{R})}=
   \sum_{d=0}^n\frac{H_d\,|t|^d}{m(m+2)\cdots(m+d-2)} \] 
with $H_d$ intrinsic functions on $M$ given by
\begin{align*}
 H_d = 
 \begin{cases}
  0 & \text{if}~d~\text{odd}, \\
  1 & \text{if}~d=0, \\
  \sum \varepsilon_{i_1\ldots i_d}^{j_1\ldots j_d}\;
   R_{i_1i_2}^{j_1j_2}\cdots R_{i_{d-1}i_d}^{j_{d-1}j_d} & \text{if}~d>0~\text{even},
 \end{cases}
\end{align*}
and
\[ R_{ij}^{kl}=H_{ij}^{kl}-H_{ji}^{kl}\;\;,\;\; 
   H_{ij}^{kl}=h_i^k\cdot h_j^l=\sum \delta_{pq}h_i^{kp}h_{j}^{lq} \]
for $i,j,k,l=1,\ldots,n$. 
In the expression for $H_d$ with $d>0$ even the sum runs over all cardinality $d$ subsets 
$\mathcal{D}$ of $\{1,\ldots,n\}$ and over all possible couplings of pairs
\[ _{i_1i_2}^{j_1j_2}\,|\, _{i_3i_4}^{j_3j_4}\,|\,\cdots\,|_{i_{d-1}i_d}^{j_{d-1}j_d}\,| \]
taken from $\mathcal{D}=\{i_1,\dots,i_d\}=\{j_1,\dots,j_d\}$. Here a pair $i_1i_2$
means two distinct numbers $i_1,i_2$ irrespective of their order.
\end{theorem}

\begin{proof}
Averaging the characteristic polynomial $\det(\lambda\delta_i^j-\sum t_ph_{i}^{jp})$ over the
orthogonal group $\mathrm{O}_m(\mathbb{R})$ acting on $t\in\mathbb{R}^m$ yields
\[ 
\langle\det(\lambda\delta_i^j-\sum t_ph_{i}^{jp})\rangle_{\mathrm{O}_m(\mathbb{R})}=
\sum_{d=0}^n\,{\lambda}^{n-d}\,|t|^d\sum_{\mathcal{D}}\,A_{\mathcal{D}}(h_i^j)
\]
with the sum over all even $d$ and all cardinality $d$ subsets $\mathcal{D}$ of 
$\{1,\ldots,n\}$ and with $A_{\mathcal{D}}(h_i^j)$ given by
\[ 
\langle\det(t\cdot h_i^j)_{i,j\in\mathcal{D}}\rangle_{\mathrm{O}_m(\mathbb{R})}=
|t|^dA_{\mathcal{D}}(h_i^j)
\]
as degree $d$ polynomial on $\mathbb{R}^{m\times d^2}$, which
is invariant under the diagonal action of $\mathrm{O}_m(\mathbb{R})$. By the first
fundamental theorem of invariant theory for $\mathrm{O}_m(\mathbb{R})$ (see 
Corollary 4.2.3 of \cite{Goodman--Wallach 1998}, which is a modern
reincarnation of Weyl's classic \cite{Weyl 1939b}) we have
\[  
A_{\mathcal{D}}(h_i^j)=B_{\mathcal{D}}(H_{ij}^{kl})
\]
with $H_{ij}^{kl}=h_i^k\cdot h_j^l$ and $i,j,k,l\in\mathcal{D}$ with $i\neq j,k\neq l$. 
From the explicit determinantal form it follows that $B_{\mathcal{D}}(H_{ij}^{kl})$ is 
in fact a linear combination of monomials of the form
\[ H_{i_1i_2}^{j_1j_2}\cdots H_{i_{d-1}i_d}^{j_{d-1}j_d} \]
with $\mathcal{D}=\{i_1,\dots,i_d\}=\{j_1,\dots,j_d\}$. Moreover under the action of the
symmetric group $\mathfrak{S}_d$ acting on both the lower and the upper indices
$B_{\mathcal{D}}(H_{ij}^{kl})$ transforms under the sign character. 
Therefore $B_{\mathcal{D}}(H_{ij}^{kl})=C_{\mathcal{D}}(R_{ij}^{kl})$ with 
$R_{ij}^{kl}=H_{ij}^{kl}-H_{ji}^{kl}$ and by symmetry for $\mathfrak{S}_d$ 
we arrive at
\[
C_{\mathcal{D}}(R_{ij}^{kl})=c(m,d)\sum \varepsilon_{i_1\ldots i_d}^{j_1\ldots j_d}\;
   R_{i_1i_2}^{j_1j_2}\cdots R_{i_{d-1}i_d}^{j_{d-1}j_d} 
\]
with the sum over all possible couplings of pairs from $\mathcal{D}$ and $c(m,d)$
a constant depending solely on $m$ and $d$. The conclusion is that
\[
\langle\det(\delta_i^j-\sum t_ph_{i}^{jp})\rangle_{\mathrm{O}_m(\mathbb{R})}=
\sum_{d=0}^n\,c(m,d)H_d\,|t|^d
\]
and all that is left is the computation of the constant $c(m,d)$.

For this computation we take the special choice $h_i^{jp}=\delta_i^j$ for $p=1$ and
$h_i^{jp}=0$ for $p\geq 2$. In that case 
\[
A_{\mathcal{D}}(h_i^j)=\frac{\int_{S^{m-1}}\,t_1^d\,d\mu(t)}{\int_{S^{m-1}}\,d\mu(t)}
\]
with $\mu$ the Euclidean measure on the unit sphere $S^{m-1}$ in $\mathbb{R}^m$. 
The integral in the numerator becomes
\[ 
\int_{-1}^1\,r^d(1-r^2)^{(m-3)/2}dr=\int_0^1\,s^{(d-1)/2}(1-s)^{(m-3)/2}ds
\]
(apart from a factor volume $\omega_{m-1}$ of $S^{m-2}$) and so equals
\[
\mathrm{B}((d+1)/2,(m-1)/2)=\frac{\Gamma((d+1)/2)\Gamma((m-1)/2)}{\Gamma((d+m)/2)}\,.
\]
In turn this implies
\[
A_{\mathcal{D}}(h_i^j)=\frac{\Gamma((d+1)/2)\Gamma(m/2)}{\Gamma(1/2)\Gamma((d+m)/2)}=
\frac{1\cdot3\cdots(d-1)}{m(m+2)\cdots(m+d-2)}\,.
\]
On the other hand $R_{ij}^{kl}=\delta_i^k\delta_j^l-\delta_j^k\delta_i^l$ and so equal to 
$\varepsilon_{ij}^{kl}$ if the pairs $ij$ and $kl$ coincide and $0$ otherwise, and hence
\[
C_{\mathcal{D}}(R_{ij}^{kl})=c(m,d)\,\frac{d!}{2^{d/2}(d/2)!}=c(m,d)\cdot1\cdot3\cdots(d-1)\,.
\]
Hence $c(m,d)=1/m(m+2)\cdots(m+d-2)$ as desired. 
\end{proof}

Recall that the volume $\omega_m$ of the unit sphere $S^{m-1}$
and the volume $\Omega_m$ of the unit ball $\mathbb{B}^m$ are related
by $\Omega_m=\omega_m/m$. The tube formula of Weyl can now be easily
derived.

\begin{corollary}\label{Weyl tube formula}
 If the domain $\mathbb{D}^m$ in $\mathbb{R}^m$ is equal to the unit ball 
 $\mathbb{B}^m$ then the tube volume is given by
\[ V_M(a)=\Omega_m\sum_{d=0}^n\frac{k_d(M)\,a^{m+d}}{(m+2)\cdots(m+d)}
   \qquad (d\;\mathrm{even}) \]
for $a>0$ small and $k_d(M)=\int_M H_d\,ds$ with $H_d$ the intrinsic
expression on $M$ in the previous theorem and $ds$ the Riemannian measure 
on $M$.
\end{corollary}

\begin{proof}
By Section~\ref{The volume of tubes} and the symmetry of $\mathbb{B}^m$ we have
\begin{align*}
   V_{U^n}(a) &=\int_{U^n}\{\int_{a\mathbb{B}^m}
   \det(\delta_i^j-\sum t_ph_{i}^{jp})\,dt\}\sqrt{\det g_{ij}}\, du \\
              &=\int_{U^n}\{\int_{a\mathbb{B}^m}
   \sum_{d=0}^n \frac{H_d\,|t|^d}{m(m+2)\cdots(m+d-2)}\,dt\}\sqrt{\det g_{ij}} \, du \\
              &=\int_{U^n}\{\omega_m \int_0^a
   \sum_{d=0}^n\frac{H_d\,r^{m+d-1}}{m(m+2)\cdots(m+d-2)}\,dr\}\sqrt{\det g_{ij}} \, du \\
              &=\Omega_m\sum_{d=0}^n\frac{\{\int_{U^n} H_d\,ds\}\,a^{m+d}}{(m+2)\cdots(m+d-2)(m+d)} 
\end{align*}
and the result follows.
\end{proof}

If we consider domains $\mathbb{D}^m$ with symmetry groups $G_m < 
\mathrm{O}_m(\mathbb{R})$ such that the invariant polynomials of degree 
$\leq n = \dim M$ for both groups agree, then we can prove 
Theorem~\ref{abstract tube formula}.

\begin{proof}[Proof of Theorem~\ref{abstract tube formula}]
By the Fubini theorem we have for $H<G$ compact groups and $f$ a continuous
function on $G$ that 
\[ \int_G\,f(g)\,d\mu_G(g)=\int_{G/H}\{\int_H\,f(gh)\,d\mu_H(h)\}\,d\mu_{G/H}(gH) \]
with $\mu_G,\mu_H$ and $\mu_{G/H}$ the normalized invariant measures on 
$G,H$ and $G/H$ respectively. Hence by the assumption on $G_m$ we have
\[ \langle\det(\delta_i^j-\sum t_ph_{i}^{jp})\rangle_{G_m}=\langle
   \det(\delta_i^j-\sum t_ph_{i}^{jp})\rangle_{\mathrm{O}_m(\mathbb{R})} \]
and so we can just argue as in the previous proof. 
\end{proof}

Using the discussion in Section~\ref{The volume of tubes}, for $n=1$ the 
tube formula is intrinsic as long as $\mathbb{D}^m$ is centrally symmetric, 
the case already covered by Hotelling~\cite{Hotelling 1939} if
$\mathbb{D}^m=\mathbb{B}^m$.

\begin{corollary}\label{Curves}
If $M$ is a curve of finite length in $\mathbb{R}^{m+1}$ and $\mathbb{D}^m$ 
is centrally symmetric, then for $a>0$ sufficiently small
\[
 V_M(a) = \mathrm{length}(M) \, \mathrm{vol}(\mathbb{D}^m) \, a^m
\]
and hence is intrinsic. \qed
\end{corollary}

The following example shows that central symmetry is not a necessary
condition.

\begin{example}
Let $\mathbb{D}^m$ be the union of half the unit ball $\{|t|\leq1,t_1\leq0\}$ 
and the cone $\{t_2^2+\dots+t_{m}^2\leq (1-t_1/b)^2,0\leq t_1\leq b\}$ with 
top $(b,0,\ldots,0)$ for some $b>0$. 
By symmetry the average of any linear function of $t_2,\ldots,t_m$ over 
$\mathbb{D}^m$ equals zero. By direct computation the average of $t_1$
over $\mathbb{D}^m$ is equal to zero if $b=\sqrt{m}$, and so the previous
corollary remains valid for this domain as well.
\end{example}

Indeed for curves it is sufficient that the center of mass of $\mathbb{D}^m$ is
at the origin by the generalized Pappus centroid theorem. See Section~\ref{Pappus type theorems}
for the higher dimensional case.

\section{Examples of polyhedral domains $\mathbb{D}^m$}
\label{Polyhedral examples}

If we are looking for domains $\mathbb{D}^m$ in $\mathbb{R}^m$ 
with a sufficiently large symmetry group $G_m<\mathrm{O}_m(\mathbb{R})$
it is natural to consider regular polytopes $\mathbb{D}^m$ in $\mathbb{R}^m$. 
It is well known that the symmetry group $G_m$ in that case is an irreducible finite 
reflection group. Such groups are classified by their Coxeter diagrams
or by letters $\mathrm{X}_m$ with $\mathrm{X}=\mathrm{A},\mathrm{B},
\mathrm{D},\mathrm{E},\mathrm{F},\mathrm{H},\mathrm{I}(k)$ for $k\geq5$.
The corresponding reflection groups are denoted by $G_m=W(\mathrm{X}_m)$.

It is a well known theorem due to Shephard and Todd 
\cite{Shephard--Todd 1954} (with a case by case proof) and Chevalley 
\cite{Chevalley 1955} (with a proof from the Book) that
the algebra of polynomial invariants for a finite reflection group 
$W<\mathrm{O}(\mathbb{R}^m)$ is itself a polynomial algebra.

\begin{theorem}\label{Chevalley theorem}
The algebra $\mathbb{R}[\mathbb{R}^m]^{W}$ of polynomial invariants
for $W$ is of the form $\mathbb{R}[p_1,\ldots,p_m]$ with $p_1,\ldots,p_m$ 
algebraically independent homogeneous invariants of degrees $d_1\,d_2,\ldots,d_m$
respectively. 
\end{theorem}

For each of the irreducible types these degrees can be calculated and 
are given in the next table. The proof of these results can be found 
in the standard text books by Bourbaki \cite{Bourbaki 1968} or by 
Humphreys \cite{Humphreys 1990}.

\bigskip
\begin{center}
\begin{tabular}{|l|l|l|} \hline
$\mathrm{type}$   & $m$      & $d_1,d_2,\ldots,d_m$ \\ \hline
$\mathrm{A}_m$    & $\geq 1$ & $2,3,\ldots,m+1$ \\ 
$\mathrm{B}_m$    & $\geq 2$ & $2,4,\ldots,2m$ \\ 
$\mathrm{D}_m$    & $\geq 4$ & $2,4,\ldots,2m-2,m$ \\
$\mathrm{E}_6$    & $6$      & $2,5,6,8,9,12$ \\ 
$\mathrm{E}_7$    & $7$      & $2,6,8,10,12,14,18$ \\
$\mathrm{E}_8$    & $8$      & $2,8,12,14,18,20,24,30$ \\ 
$\mathrm{F}_4$    & $4$      & $2,6,8,12$ \\ 
$\mathrm{H}_3$    & $3$      & $2,6,10$ \\
$\mathrm{H}_4$    & $4$      & $2,12,20,30$ \\ 
$\mathrm{I}_2(k)$ & $2$      & $2,k \geq 5$ \\ \hline
\end{tabular}
\end{center}

\bigskip
So for $m\geq 2$ the irreducible finite reflection group
$W(\mathrm{X}_m)<\mathrm{O}(\mathbb{R}^m)$ is orthogonal of degree $d_2-1$
in the sense of Definition{\;\ref{orthogonal of degree n definition}}.

\begin{corollary}\label{polyhedral tube formula}
If $\mathbb{D}^m$ is a domain in $\mathbb{R}^m$ invariant under a finite 
reflection group $W(\mathrm{X}_m)$ then the tube formula of 
Theorem~\ref{abstract tube formula} does hold with intrinsic coefficients 
if $n=\dim M < d_2$, that is, the dimension $n$ of $M$ is strictly smaller 
than the second fundamental degree $d_2$. \qed
\end{corollary}

For example, if $\mathbb{D}^3$ is an icosahedron with symmetry group
$W(\mathrm{H}_3)$ then the tube formula is intrinsic for submanifolds $M$ 
of dimension $n\leq5$ in $\mathbb{R}^{n+3}$, and if $\mathbb{D}^4$ is a 
$600$-cell with symmetry group $W(\mathrm{H}_4)$ then the tube formula is 
intrinsic for $n\leq11$. 
For any dimension $n$ of $M\hookrightarrow\mathbb{R}^{n+2}$ with 
$\mathbb{D}^2$ a regular $k$-gon with $k>n$ the tube formula is intrinsic, 
since its symmetry group is $W(\mathrm{I}_2(k))$.
For dimension $n=2$ or $3$ we find in this way examples of intrinsic
tube formulas for arbitrary codimension $m$ via domains $\mathbb{D}^m$ with 
symmetry groups $W(\mathrm{A}_m)$ ($n=2$) and $W(\mathrm{B}_m)$ ($n=2,3$), 
respectively. However, for dimension $n\geq 4$ we obtain in this way 
only examples of intrinsic tube formulas with relatively small codimension 
$m\leq8$.

\medskip
Examples with larger codimension $m$ can be obtained by the following construction.

\begin{corollary}
Let $G$ be a noncompact simple Lie group acting on its Lie algebra $\mathfrak{g}$, and let 
$\theta$ be a Cartan involution of $G$ and $\mathfrak{g}$ and $\mathfrak{g} = \mathfrak{k} 
\oplus \mathfrak{p}$ the decomposition in $+1$ and $-1$ eigenspaces of $\theta$ on 
$\mathfrak{g}$. If the domain $\mathbb{D}\subset\mathfrak{p}$ is the convex hull of a 
nonzero orbit of $K =G^\theta$ on $\mathfrak{p}$ then the tube formula of 
Theorem~\ref{abstract tube formula} does hold with intrinsic coefficients under the 
assumption that the dimension $n$ of $M\hookrightarrow\mathbb{R}^{n+m}$ (with
$m=\dim\mathfrak{p}$) is strictly smaller than the second fundamental 
degree $d_2$ of the Weyl group $W$ of the pair $(\mathfrak{g},\theta)$.
\end{corollary}

\begin{proof}
The Killing form $(\cdot,\cdot)$ on $\mathfrak{p}$
is positive definite and the fixed point group $K=G^{\theta}$ of
$\theta$ on $G$ acts on $\mathfrak{p}$ as a subgroup of 
$\mathrm{SO}(\mathfrak{p})$. If $\mathfrak{a}\subset\mathfrak{p}$ is
a maximal Abelian subspace then each orbit of $K$ on $\mathfrak{p}$
intersects $\mathfrak{a}$ in an orbit of the Weyl group 
$W=\mathrm{N}_K(\mathfrak{a})/\mathrm{Z}_K(\mathfrak{a})$ of the
pair $(\mathfrak{g},\theta)$. Hence each invariant polynomial
$p\in\mathbb{R}[\mathfrak{p}]^K$ for $K$ on $\mathfrak{p}$ restricts
to a Weyl group invariant polynomial on $\mathfrak{a}$. It is a theorem
of Chevalley (see Lemma $7$ in \cite{Harish-Chandra 1958}) that the
restriction map 
\[ \mathbb{R}[\mathfrak{p}]^K\rightarrow
   \mathbb{R}[\mathfrak{a}]^W \]
is an isomorphism of algebras. Since $W$ acts on $\mathfrak{a}$ as a 
finite reflection group the latter algebra is described by Theorem
{\ref{Chevalley theorem}}. The possible finite reflection groups that
can occur as such a Weyl group $W$ are those reflection groups, 
which can be defined over $\mathbb{Z}$. This means that $\mathrm{H}_3$ 
and $\mathrm{H}_4$ are excluded and only the dihedral types 
$\mathrm{I}_2(k)=\mathrm{A}_2,\mathrm{B}_2,\mathrm{G}_2$ for 
$k=3,4,6$ respectively are allowed. 
The text books \cite{Helgason 1978} and \cite{Helgason 1984} by Helgason
give a thorough exposition of the theory. Using the convexity theorem of 
Kostant \cite{Kostant 1973} it is easy to see that the convex hull
of an orbit of $K$ on $\mathfrak{p}$ intersects $\mathfrak{a}$
in the convex hull of an orbit of $W$ on $\mathfrak{a}$.
\end{proof}

For example, if $G$ is the complex Lie group of type $\mathrm{E}_8$ 
(and so $K$ is the compact Lie group of type $\mathrm{E}_8$ acting on
$\mathfrak{p}=i\mathfrak{k}$) then
we do find in this way examples of local submanifolds $M$ of Euclidean 
space of dimension $n \leq 7$ and of codimension $m=248$ for which
the tube formula of Theorem~\ref{abstract tube formula} has 
intrinsic coefficients. Presumably this large codimension relative to 
the small dimension of $M$ allows for an abundance of room for isometric 
deformations for the embedding of $M$ in such a Euclidean space.

\section{No-go results for diamond domains $\Dd^m$}
\label{No go results for diamant domains}

In this section we shall denote by $\Dd^m$ the convex hull of the subset
\[ \{t_1^2+\ldots+t_{m-1}^2\leq1,t_m=0\}\sqcup\{(0,\ldots,0,\pm1)\} \]
in $\mathbb{R}^m$, $m\geq 2$. 
Any multiple $a\Dd^m$ for $a>0$ will be called a diamond domain. 
The symmetry group $G_m$ of $\Dd^m$ is equal
to $\mathrm{O}_{m-1}(\mathbb{R})\times\mathrm{O}_1(\mathbb{R})$ for 
$m\geq 3$ while for $m=2$ the symmetry group $G_2$ is equal to the 
dihedral group $W(\mathrm{B}_2)$ of order $8$.
The essential point of Weyl's argument for the intrinsic nature of the 
volume formula for tubes is the computation of the integral
\[ \int_{a\Dd^m}\det(\delta_i^j-\sum t_ph_{i}^{jp})\,dt \]
as a polynomial in $a$, by first averaging over the symmetry group 
$G_m$ of the domain $\Dd^m$ in $\mathbb{R}^m$. The outcome should
hopefully be a polynomial expression in Riemann curvature components 
$R_{ij}^{kl}$ as in Theorem~\ref{integrand average theorem}.
We will work out two examples, one with $n=2$ and $m\geq3$ and the other
with $n=4$ and $m=2$, where this does not work.

\begin{example}\label{example n=2, m>2}
Let us first consider the case of surfaces of codimension $m$ at least $3$. 
Since the symmetry group of $\Dd^m$ is then $G_m = \mathrm{O}_{m-1}(\mathbb{R})
\times \mathrm{O}_1(\mathbb{R})$ the invariants of degree $2$ in $\mathbb{R}
[t_1,\ldots,t_m]$ are linear combinations of $R=(t_1^2+\ldots+t_{m-1}^2)/(m-1)$ 
and $S=t_m^2$. The above determinant for $n=2$ becomes
\[ \det(\delta_i^j-\sum t_ph_{i}^{jp}) = 
   1-\sum t_p(h_{1}^{1p}+h_{2}^{2p})+\sum t_pt_q
   (h_{1}^{1p}h_{2}^{2q}-h_{1}^{2p}h_{2}^{1q}) \]
and averaging over the symmetry group $G_m$ yields
\[ 1+A(h_{i}^j)R(t)+B(h_{i}^j)S(t) \]
with
\[
A = \sum_{p=1}^{m-1}(h_{1}^{1p}h_{2}^{2p}-h_{1}^{2p}h_{2}^{1p}) \qquad
B = (h_{1}^{1m}h_{2}^{2m}-h_{1}^{2m}h_{2}^{1m})
\]
and $A+B= R_{12}^{12}$ intrinsic. Thus by the above, if the integrals of $R$ and 
$S$ over $\widehat{\mathbb{D}}^m$ agree, then the generalized tube volume is intrinsic as well.
The integrals of $R(t)$ and $S(t)$ over $\Dd^m$ amount respectively to (put 
$r=\sqrt{R}$ and $s=\sqrt{S}$)
\[ \frac{2\,\omega_{m-1}}{m-1}\int r^2r^{m-2}\,dr\,ds\quad
   \mathrm{and}\quad 2\,\omega_{m-1}\int s^2 r^{m-2}\,dr\,ds, \]
integrated over the triangle $\{(r,s)\,;r,s\geq0,r+s\leq1\}$,  
and we will show that for $m\geq3$ these are distinct. Apart from
the factor $2\,\omega_{m-1}$ the left integral becomes
\[ \frac{1}{m-1}\int_0^1(1-r)r^m\,dr=\frac{1}{(m-1)(m+1)(m+2)} \]
while the right integral equals
\[ \int_0^1\tfrac13(1-r)^3r^{m-2}\,dr=\frac{2}{(m-1)m(m+1)(m+2)} \]
and for the difference we find
\[ \frac{m-2}{(m-1)m(m+1)(m+2)} \]
which is nonzero for $m\geq3$, as claimed. Hence the tube volume formula for
a general surface $M$ in $\mathbb{R}^{2+m}$ with diamond domain 
$\widehat{\mathbb{D}}^m$ is no longer intrinsic for $m\geq3$. For $m=2$ it still is 
intrinsic as should, because $G_2=W(\mathrm{B}_2)$ is orthogonal of degree $3$
(in fact, it is not only intrinsic for $n=2$ but also for $n=3$ since odd exponents 
vanish for the centrally symmetric diamond $\Dd^m$).
\end{example}

\begin{example}\label{example n=4, m=2}
Let us next consider the case that $n=4$ and $m=2$. The symmetry group of the diamond
domain $\widehat{\mathbb{D}}^2=\{(t_1,t_2) \, ; |t_1|+|t_2|\leq1\}$ is the dihedral group 
$W(\mathrm{B}_2)$ of order $8$ generated by the two reflections 
$s_1(t_1,t_2)=(-t_1,t_2)$ and $s_{2}(t_1,t_2)=(t_2,t_1)$. The invariant 
polynomials for this group $W(\mathrm{B}_2)$ are generated as an algebra 
by the quadratic invariant $P(t)=t_1^2+t_2^2$ and the quartic invariant 
$Q(t)=t_1^2t_2^2$. Hence any quartic invariant is a unique linear combination 
of $Q$ and $R(t)=t_1^4+t_2^4=P^2-2Q$.

We would like to know if Weyl's averaging trick (over the dihedral group 
$W(\mathrm{B}_2)$ this time) remains valid for any pencil of second 
fundamental forms.
In order to keep the calculation as simple as possible we look at the
special case that $h_{i}^{jp}=0$ for $i\neq j$ and $p=1,2$. If we
write $h_{i}^{i1}=a_i$ and $h_{i}^{i2}=b_i$ we get
\[ \det(\delta_i^j-\sum t_p h_{i}^{jp})=\prod_{i=1}^4\;(1-t_1a_i-t_2b_i) \]
and averaging over the dihedral group $W(\mathrm{B}_2)$ yields
\[ 1+A(a,b)P(t)+B(a,b)Q(t)+C(a,b)R(t) \]
with $A,B,C$ homogeneous polynomials of degree $2,4,4$ respectively.
A direct calculation gives
\begin{align*}
A &= \tfrac12\sum_{i<j}\;(a_ia_j+b_ib_j) = \tfrac12 \sum_{i<j} R_{ij}^{ij}\\
B &= \sum_{i<j,k<l}\;a_ia_jb_kb_l = \sum_{i<j,k<l} R_{ij}^{ij} R_{kl}^{kl} - 
3 (a_1a_2a_3a_4 + b_1b_2b_3b_4)  \\
C &= \tfrac12(a_1a_2a_3a_4+b_1b_2b_3b_4)
\end{align*}
with $\{i,j,k,l\}=\{1,2,3,4\}$ in the sum for $B$ and $R_{ij}^{ij}=h_i^i\cdot h_j^j$
as before. Note that $A$ as well as $B+6C$ are intrinsic quantities. 
For the integrals of $Q$ and $R$ over $\widehat{\mathbb{D}}^2$ we find 
(put $r=|t_1|$ and $s=|t_2|$)
\[
 \int Q(t) \, dt_1 \, dt_2 = 4\int_0^1 \{\int_0^{1-r} r^2 s^2 \, ds\} \, dr = \tfrac{1}{45}
\]
and
\[
 \int R(t) \, dt_1 \, dt_2 = 4\int_0^1 \{\int_0^{1-r} (r^4 + s^4) \, ds\} \, dr = \tfrac{4}{15} 
 \neq \tfrac{6}{45}.
\]
Hence for fourfolds in $\mathbb{R}^6$ with diamond domain $\Dd^2$ we see that 
the tube volume formula need no longer be intrinsic.
\end{example}

The conclusion therefore is that the tube formula for submanifolds $M$ in
$\mathbb{R}^{n+m}$ of dimension $n$ with cross section the diamond
$\widehat{\mathbb{D}}^m$ will in general no longer be intrinsic, unless we
are in one of the cases of the following table.
\bigskip
\begin{center}
\begin{tabular}{|l|l|l|} \hline
$m=\codim M$ & symmetry group of the diamond in $\mathbb{R}^{m}$ & $n=\dim M$  \\ \hline
$1$ & $\mathrm{O}_1(\mathbb{R})$ &  any \\
$2$ & $W(\mathrm{B}_2)$ & $\leq 3$ \\ 
any & $\mathrm{O}_{m-1}(\mathbb{R}) \times \mathrm{O}_{1}(\mathbb{R})$ & $1$  \\ \hline
\end{tabular}
\end{center}
\bigskip
Our motivation for looking at diamond tubes in a Euclidean vector space came
from the analogous causal tubes in a Lorentzian vector space, which are 
discussed in the next section.

\section{Riemannian submanifolds of a Lorentzian vector space}
\label{Riemannian submanifolds of a Lorentzian vector space}

Let us suppose that $M$ is a compact connected $n$-dimensional
Riemannian submanifold of an ambient Cartesian space $\mathbb{R}^{n+m}$, 
equipped with a nondegenerate but possibly indefinite 
scalar product denoted by a dot. Let $\mathbb{D}^m$ be a compact domain
around $0$ in $\mathbb{R}^m$.
Say we have a local parametrization around $M$ given by
\[ x:U^n\times \mathbb{D}^m\rightarrow\mathbb{R}^{n+m},\quad
   (u,t)\mapsto x(u,t)=r(u)+\sum t^pn_p(u) \]
with $u=(u^1,\ldots,u^n)\in U^n$, $t=(t^1,\ldots,t^m)\in \mathbb{D}^m$ while 
$n_1(u),\ldots,n_m(u)$ are vectors in $\mathbb{R}^{n+m}$ depending smoothly on
$u\in U^n$ and
\[ \partial_ir(u)\cdot n_p(u)=0,\quad n_p(u)\cdot n_q(u)=\eta_{pq} \]   
for all $u\in U^n$, all $i=1,\ldots,n$, all $p,q=1,\ldots,m$ and $\eta_{pq}$
a $m\times m$ diagonal matrix with entries $\pm1$ (so in particular constant, that is independent
of $u\in U^n$). Observe that the choice of such an orthonormal frame for the normal
bundle of $M$ in $\mathbb{R}^{n+m}$ is in principle only possible locally. Indeed if 
$0\in U^n$ then by linear algebra we can choose a basis $n_1(0),\ldots,n_m(0)$ for the
orthogonal complement of the tangent vectors $\partial_1r(0),\ldots,\partial_nr(0)$ with
$n_p(0)\cdot n_q(0)=\eta_{pq}$ and subsequently apply Gram--Schmidt to the
vectors $\partial_1r(u),\ldots,\partial_nr(u),n_1(0),\ldots,n_m(0)$ for $u$ small.

As in Section~\ref{The volume of tubes} we can write
\[ \partial_ix=\sum_j(\delta_i^j-\sum t^pn_p\cdot h_i^j)\partial_jr+\ldots \]
with second fundamental form normal vectors $h_i^j=\sum g^{jk}h_{ik}$ and 
the dots $\ldots$ stand for a linear combination of the normal fields $n_p$.
Likewise writing $t_p=\sum \eta_{pq}t^q$ we arrive at the generalized tube volume formula
\[ 
V_{U^n}(a)=\int_{U^n}\{\int_{a\mathbb{D}^m}\det(\delta_i^j-
\sum t_ph_i^{jp})\,dt\}\sqrt{\det g_{ij}}\,du 
\]
with $\mathbb{D}^m$ a compact domain around $0$ and $a>0$ sufficiently small.
Hence $V_{M}(a)$ is a polynomial in $a$ of degree $m+n$ with 
$\mathrm{vol}(M)\mathrm{vol}(\mathbb{D}^m)a^m$ as lowest order term.
The Gauss equations 
\[ R_{ij}^{kl}=h_i^k\cdot h_j^l-h_j^k\cdot h_i^l \]
as derived in Section~\ref{The Gauss equations} remain valid for an indefinite scalar product.

\medskip
For Riemannian curves $M$ of dimension $n=1$ and a centrally symmetric domain
$\mathbb{D}^m$ around $0$ in $\mathbb{R}^m$ we get
\[ V_{M}(a)=\mathrm{length}(M)\mathrm{vol}(\mathbb{D}^m)a^m \]
just like the original case of Hotelling~\cite{Hotelling 1939}.
Also, if $\eta_{pq}=\delta_{pq}$ then we are essentially in the original setting of Weyl and 
his spherical tube formula and our variations hold without change.

\medskip
Let us suppose for the rest of this section that $M$ is a compact Riemannian
submanifold of a Lorentzian vector space $\mathbb{R}^{n+m-1,1}$ with scalar 
product $\cdot$ of signature $(n+m-1,1)$ and thus 
$\eta_{pq}=\mathrm{diag}(1,\ldots,1,-1)$ in $\mathbb{R}^m$. 
If we denote by $\mathbb{J}=\{x\in \mathbb{R}^{n+m-1,1}\,;x\cdot x\leq0\}$ the 
causal future and past of the origin then for $e$ a unit timelike vector the domain
$\widehat{\mathbb{D}}^{n+m}(e)=\{e+\mathbb{J}\}\cap\{-e+\mathbb{J}\}$ is 
called the \emph{causal diamond} around $0$ with unit timelike normal $e$. It is the locus 
traced out by all causal curves between $e$ and $-e$. Any two causal diamonds 
around $0$ can be transformed into each other by an element of the 
Lorentz group $\mathrm{O}_{n+m-1,1}(\mathbb{R})$, while the symmetry group of a 
causal diamond is isomorphic to 
$\mathrm{O}_{n+m-1}(\mathbb{R})\times\mathrm{O}_1(\mathbb{R})$. The set
\[ \{r+n\,;r\in M,\,n\in N_rM\cap a\widehat{\mathbb{D}}^{n+m}(n_m(r))\} \]
will be called the \emph{causal tube} with radius $a>0$ (sufficiently small) around 
$M$ relative to the unit timelike normal field $n_m$. Its volume is given by
\[ 
V_M(a)=\int_M\{\int_{a\widehat{\mathbb{D}}^m}\det(\delta_i^j-\sum t_ph_i^{jp})\,dt\}\,ds
\]
with $\widehat{\mathbb{D}}^m$ the diamond domain in $\mathbb{R}^m$ in the notation
of the previous section.

In accordance with Weyl's tube formula, apart from the $\pm$ sign, we obtain the 
following version of the tube formula for Riemannian hypersurfaces.

\begin{corollary}
For a spacelike hypersurface $M$ of codimension $m=1$ in a Lorentzian vector space
$\mathbb{R}^{n,1}$ the causal tube volume formula takes the form
\[ V_{M}(a)=2\sum_{d=0}^n \frac{(-1)^{d/2}k_d(M)a^{1+d}}
   {3\cdot5\cdots(1+d)}\qquad (d\;\mathrm{even}). \]
\end{corollary}

Indeed if $h_{ij}$ is the scalar valued second fundamental form then 
$H_{ij}^{kl}=-h_i^kh_j^l$
and so $R_{ij}^{kl}$ in Theorem~\ref{integrand average theorem} also picks up a minus sign,
that is $H_d$ and $k_d(M)=\int_M H_d\,ds$ pick up a factor $(-1)^{d/2}$.

There is yet another case, where the causal tube formula has an intrinsic form, namely 
in case $M\hookrightarrow\mathbb{R}^{n+m-1}\hookrightarrow\mathbb{R}^{n+m-1,1}$.
This can be checked easily using Weyl's tube formula in a straightforward way.

The next example shows, however, that the positive result for diamond tubes for 
$\dim M = \codim M = 2$ of Section~\ref{No go results for diamant domains} cannot 
be extended to the Lorentzian setting.

\begin{example}\label{example n=2, m=2}
If we specialize to the case $n=m=2$ and $\eta_{pq}=\mathrm{diag}(1,-1)$ of a 
compact spacelike surface $M$ in Minkowski spacetime $\mathbb{R}^{3,1}$ then the integrand
\[ 
\det(\delta_i^j-\sum t_ph_{i}^{jp}) = 
1-\sum t_p(h_{1}^{1p}+h_{2}^{2p})+\sum t_pt_q
(h_{1}^{1p}h_{2}^{2q}-h_{1}^{2p}h_{2}^{1q}) 
\]
averages over the symmetry group $W(\mathrm{B}_2)$ of the square 
$\widehat{\mathbb{D}}^2$ as in Example~\ref{example n=2, m>2} to
the expression
\[ 
1+(A(h_i^j)+B(h_i^j))(t_1^2+t_2^2)/2
\]
with $A=h_1^{11}h_2^{21}-h_1^{21}h_2^{11}$ and $B=h_1^{12}h_2^{22}-h_1^{22}h_2^{12}$.
Since
\[ \int_{a\widehat{\mathbb{D}}^2}dt_1\,dt_2=2a^2,
   \quad\int_{a\widehat{\mathbb{D}}^2}t_1^2\,dt_1\,dt_2=
    \int_{a\widehat{\mathbb{D}}^2}t_2^2\,dt_1\,dt_2=a^4/3 \]
we find
\[ V_M(a)=\int_M\{2a^2+(A+B)a^4/3\}\,ds =
\mathrm{area}(M)\,2a^2+\int_M(A+B)\,ds\,a^4/3 \]    
for the volume of the causal tube along $M$.

On the other hand, the Gauss equation (for n=2 there is just a single one) in this
particular case of spacelike surfaces in Minkowski spacetime becomes
\[ R_{12}^{12}=h_1^1\cdot h_2^2-h_1^2\cdot h_2^1=A-B. \]
Since $\int_M(A+B)\,ds$ enters in the tube volume formula while 
$\int_M(A-B)\,ds$ is the total Gauss curvature for $M$, the volume formula 
for causal tubes around surfaces need not be intrinsic. 
\end{example}

The conclusion is that for spacelike submanifolds of Minkowski spacetime 
$\mathbb{R}^{3,1}$ the causal tube volume formula will in general no longer 
be intrinsic, except for the obvious cases of spacelike curves ($n=1$) or 
hypersurfaces ($m=1$). This question about the intrinsic nature of causal tube 
volume formulas was the starting point for our work.

\section{Pappus type theorems}\label{Pappus type theorems}

Let us denote the graded commutative algebra $\mathbb{R}[t_1,\ldots,t_m]$
by $P=\oplus\,P^d$. The subalgebra of invariants for $\mathrm{O}_m(\mathbb{R})$ 
is equal to $\mathbb{R}[t_1^2+\ldots+t_m^2]$ and is denoted $I=\oplus\,I^d$. 
The graded subspace 
\[ C=\{p\in P\,;\int_{\mathrm{O}_m(\mathbb{R})}\,g(p)\,d\mu(g)=0\}=\oplus\,C^d \]
is the unique invariant complement of $I$ in $P$. Here $\mu$ is the normalized Haar 
measure on $\mathrm{O}_m(\mathbb{R})$. Hence $P=I\,\oplus\,C$ and clearly
$C^d=P^d$ for $d$ odd while $C^d$ has codimension one in $P^d$ for $d$ even.

\begin{definition}
A compact domain $\mathbb{D}^m$ in $\mathbb{R}^m$ is called 
\emph{symmetric of degree $n$} if 
\[ \int_{\mathbb{D}^m}\,p(t)\,dt=0 \]
for all polynomials $p\in C^1\oplus\ldots\oplus C^n$.
\end{definition}
If the compact domain $\mathbb{D}^m$ has a symmetry group
$G_m$ that is orthogonal of degree $n$ (in the sense of our 
Definition~\ref{orthogonal of degree n definition}) then
\[ 
\int_{\mathbb{D}^m}\,p(t)\,dt=\int_{\mathbb{D}^m}\,\langle p(t)\rangle_{G_m}\,dt=
\int_{\mathbb{D}^m}\,\langle p(t)\rangle_{\mathrm{O}_m(\mathbb{R})}\,dt 
\]
for all polynomials $p(t)$ of degree $\leq n$. In particular, if the symmetry group 
$G_m$ of $\mathbb{D}^m$ is orthogonal of degree $n$ then the domain 
$\mathbb{D}^m$ is necessarily symmetric of degree $n$. From the discussions 
in Section~\ref{The volume of tubes} and Section~\ref{Averaging the integrand}
it follows that our Theorem~\ref{abstract tube formula} holds with the condition on
the symmetry group $G_m$ of $\mathbb{D}^m$ being orthogonal of degree $n$
replaced by the condition on $\mathbb{D}^m$ being symmetric of degree $n$.
This more general form of Theorem~\ref{abstract tube formula} was obtained 
as Theorem~4.4 in \cite{Domingo-Juan--Miquel 2004}.

A compact domain $\mathbb{D}^m$ in $\mathbb{R}^m$ is symmetric of degree 
$1$ if and only if the center of mass of $\mathbb{D}^m$ lies at the origin. 
Hence the condition for $\mathbb{D}^m$ to be symmetric of degree $1$ is a good 
deal more general than the condition for the symmetry group $G_m$ to be orthogonal 
of degree $1$. If the manifold $M$ is a circle in $\mathbb{R}^3$ then the tube volume 
formula boils down to the ancient Pappus's centroid theorem. For this reason the higher 
dimensional tube volume formulas are sometimes also called Pappus type theorems.

The next example shows that for a planar domain $\mathbb{D}^2$ and
for all $n\geq 1$ the notion for $\mathbb{D}^2$ to be symmetric of degree $n$ 
is strictly weaker than the notion for the symmetry group $G_2$ of $\mathbb{D}^2$ 
being orthogonal of degree $n$.

\begin{example}
Consider in polar coordinates $t_1=r\cos\phi,t_2=r\sin\phi$ the planar domain
$\mathbb{D}^2=\{(r,\phi)\,;0\leq r\leq a(\phi),\phi\in\mathbb{R}/2\pi\mathbb{Z}\}$ 
for some continuous function $a\colon\mathbb{R}/2\pi\mathbb{Z}\rightarrow(0,\infty)$.
The space $C^d$ is spanned by the functions $r^d\cos(e\phi)$ and $r^d\sin(e\phi)$ 
with $1\leq e\leq d$ and $e\equiv d$ (mod $2$). The condition that $\mathbb{D}^2$
is symmetric of degree $n$ amounts to
\[
\int_0^{2\pi}\int_0^{a(\phi)}\,r^{d+1}\,dr\cos(e\phi)\,d\phi=
\int_0^{2\pi}\int_0^{a(\phi)}\,r^{d+1}\,dr\sin(e\phi)\,d\phi=0
\]
or equivalently
\[
\int_0^{2\pi}(a(\phi))^{d+2}\cos(e\phi)\,d\phi=
\int_0^{2\pi}(a(\phi))^{d+2}\sin(e\phi)\,d\phi=0
\]
for all $1\leq d\leq n$, $1\leq e\leq d$ and $e\equiv d$ (mod $2$). Clearly these
conditions are satisfied if for some $k>n$ the function $a(\phi)$ is invariant under 
the cyclic group $C_k$ of order $k$ acting on the circle $\mathbb{R}/2\pi\mathbb{Z}$ 
by rotations. Indeed, in that case the Fourier coefficients of all functions $a(\phi)^{d+2}$
vanish for modes not contained in $k\mathbb{Z}$. This is in accordance with our
Theorem~\ref{abstract tube formula} since the symmetry group $C_k$ of this domain $\mathbb{D}^2$
is orthogonal of degree $k>n$.

However, if for a fixed $n\geq1$ one chooses integers $p>n$ and $q>(n+3)p$ then the
function $a(\phi)=b(\phi)(2+\cos(p\phi))$ with $b>0$ invariant under $C_q$ has
the property that the Fourier coefficients of all functions $a(\phi)^{d+2}$ for 
$1\leq d\leq n$ vanish for modes $\pm1,\dots,\pm n$. Hence this domain $\mathbb{D}^2$ 
is certainly symmetric of degree $n$. On the other hand, if we pick $p$ and $q$ 
relatively prime then the symmetry group $G_2$ of $\mathbb{D}^2$ will be trivial in 
case $b(\phi)$ is chosen sufficiently general (so that the symmetry group for 
$b(\phi)$ is not larger than $C_q$), and $G_2=\{1\}$ is not orthogonal of any 
degree $n\geq 1$. 
\end{example}

The examples obtained in Proposition~4.3 of \cite{Domingo-Juan--Miquel 2004} of compact domains 
$\mathbb{D}^m$ in $\mathbb{R}^m$ that are symmetric of degree $n$
are for $n\geq2$ domains $\mathbb{D}^2$ with dihedral symmetry and
for $n=2,3$ domains $\mathbb{D}^m$ with hyperoctahedral symmetry, besides
of course the unit ball $\mathbb{B}^m$ for all $n$. 
Hence apart from giving a pedestrian exposition of Weyl's tube volume formula and also
a discussion of tube volume formulas for Riemannian submanifolds of a Lorentzian vector 
space our paper gives a more complete and transparent discussion in 
Section~\ref{Polyhedral examples} of examples based on symmetry of cross sections 
$\mathbb{D}^m$ for which the intrinsic tube volume formula holds.

\end{document}